\documentclass[a4paper,10pt]{article}
\usepackage[margin=2.5cm]{geometry}
\usepackage{amsmath,amsfonts,enumerate,amsthm,color, amssymb}

\usepackage{pgf}
\newtheorem{lemma}{Lemma}[section]
\newtheorem{proposition}[lemma]{Proposition}
\newtheorem{corollary}[lemma]{Corollary}
\newtheorem{theorem}[lemma]{Theorem}

\newtheorem{example}[lemma]{Example}

\newtheorem{remark}[lemma]{Remark}

\newcommand{\ZZ}{\mathbb{Z}}

\renewcommand{\P}{\mathcal{P}}

\newcommand{\comment}[1]{}

\DeclareMathOperator{\Cay}{Cay}

\DeclareMathOperator{\Ml}{ML}

\title{\bf \Large On $k$-rainbow domination in regular graphs} 
\author{
{\sc Bo\v stjan Kuzman}}
\date{University of Ljubljana, Faculty of Education,\\
and Institute of Mathematics, Physics and Mechanics (IMFM)}

\begin{document}

\maketitle

\begin{abstract}
The $k$-rainbow domination problem is studied for regular graphs. We prove that the $k$-rainbow domination number $\gamma_{rk}(G)$ of a $d$-regular graph for $d\leq k\leq 2d$ is bounded below by $\displaystyle{\left\lceil kn/2d\right\rceil}$, where $n$ is the order of a graph. We determine necessary conditions for regular graphs to attain this bound and
find several examples. As an application, we determine exact $k$-rainbow domination numbers for all cubic Cayley graphs over abelian groups.\medskip



\end{abstract}

\footnotetext{
This work is supported in part by the Slovenian Research Agency (ARRS), research program P1-0285 and research projects J1-9108, J1-1694, J1-1695.
\\ [+1ex] {\em Corresponding author e-mail:} bostjan.kuzman@pef.uni-lj.si (B. Kuzman)}

\section{Introduction}
\label{sec:intro}

The concept of $k$-rainbow domination, as introduced by Bre\v sar et al. in~\cite{Bresar H R 2008}, is an extension of the classical domination problem in graphs and was initiated by Hartnell and Rall~\cite{Hartnell Rall 2004}, who studied the domination numbers of cartesian product $G\square K_2$ in relation to the Vizing conjecture. 
The problem has natural applications in analysis of networks.\medskip

Let $G=(V,E)$ be a finite, simple graph, and let $k$ be a given nonnegative integer. We denote by ${\cal C}=\{1,\ldots,k\}$ the set of \emph{colors}. We say that a coloring function $f\colon V\to 2^{\cal C}$ that assigns to each vertex $v\in V$ a subset of colors $f(v)\subseteq {\cal C}$, is a \emph{$k$-rainbow dominating function} (or $k$-RDF) on $G$, if
$$\forall v\in V\colon f(v)=\emptyset \implies \bigcup_{u\sim v}f(u)={\cal C}.$$
In other words, every non-colored vertex $v\in V$ is $k$-rainbow dominated by its neighbours of all possible colors.
For a given $k$-RDF, we define the \emph{weight} of $f$ as 
$$w(f)=\sum_{v\in V}|f(v)|.$$
The \emph{$k$-rainbow domination number} $\gamma_{rk}(G)$ is the minimal possible weight attained by a $k$-RDF on $G$:
$$\gamma_{rk}(G)=\min\{w(f)\mid f\colon V\to2^{\cal C}\text{ is a $k$-RDF}\}.$$
Any $k$-rainbow domination function $f$ of weight $w(f)=\gamma_{rk}(G)$ is called a \emph{$\gamma_{rk}(G)$-function}. 
\medskip

Clearly, any $1$-rainbow domination function on $G$ corresponds to a usual dominating set for $G$, and $1$-rainbow domination number $\gamma_{r1}(G)$ coincides with the usual domination number $\gamma(G)$ of the graph. Also, it was shown in \cite{Bresar H R 2008}, that $\gamma_{rk}(G)=\gamma(G\square K_k)$, where $G\square K_k$ denotes the cartesian product of $G$ with a complete graph on $k$ vertices.
For given graph $G$ and positive integer $k$, determination of the exact value $\gamma_{rk}(G)$ is known to be NP-complete~\cite{Bresar Kraner 2007, Chang 2009}.  

For $k=2,3$, the exact values and upper bounds for $k$-rainbow domination numbers of specific graph families such as the generalized Petersen graphs, trees, products of paths and cycles, grid graphs, etc, were studied in several papers, see for instance~\cite{Bresar Kraner 2007, Chang 2009, Kraner Rall Tepeh lexi 2013, Shao 2014, Stepien 2014, Wang 2013}. 
Moreover, general upper bounds for $k$-rainbow domination numbers of connected graphs are also known and were proven to be tight for some specific graphs, see~\cite{Fujita 2015, Furuya 2018}.  
%
%$\gamma_{r1}(G)\leq n/2$, $\gamma_{r2}(G)\leq 3n/4$, $\gamma_{r3}(G)\leq 8n/9$, and $\gamma_{r4}(G)\leq n$, for any connected graph $G$, and if $G$ has no vertices of degree $1$, then $\gamma_{r1}(G)\leq 2n/5$, $\gamma_{r2}(G)\leq 2n/3$, $\gamma_{r3}(G)\leq 5n/6$, see

However, there are few results available on determining the full set of $k$-rainbow dominaton numbers  of particular graph families for all relevant $k$. In order to determine these for some graph $G$, one could first determine $\gamma_{r1}(G)$ and then apply the following theorem to obtain some lower and upper bounds for $\gamma_{rk}(G)$ for $k\geq 2$.

\begin{theorem}[Shao et al., 2014]\label{thm: t<t'}\label{thm: rk basic} 
Let $G=(V,E)$ be a connected graph of order $n$.
\begin{enumerate}[(i)]
\item If $k'>k$, then $$\gamma_{rk'}(G)\leq \gamma_{rk}(G)+(k'-k)\left\lfloor \frac{\gamma_{rk}(G)}{k}\right\rfloor.$$
\item If $\Delta$ denotes the maximal degree in graph $G$, then
$$\gamma_{rk}(G)\geq \left\lceil \frac{kn}{\Delta+k}\right\rceil.$$ 
\end{enumerate}
\end{theorem}
In this paper, we focus on $k$-rainbow domination numbers of regular graphs. In Section~\ref{sec: reg}, we prove the following Theorem, which improves the lower bound of Theorem \ref{thm: rk basic} for $d$-regular graphs, whenever $d\leq k\leq 2d$.
\begin{theorem}\label{thm: gamma rk pred}
Let $G$ be a 
$d$-regular graph of order $n$.  Then
\begin{align}
\gamma_{rk}(G)&\geq \left\lceil \frac{kn}{2d}\right\rceil \text{ for }k\leq 2d
\end{align}
and $\gamma_{rk}(G)=n \text{ for }k\geq 2d$.
\end{theorem}
In Section~\ref{sec: reg}, we also prove several other inequalities and interesting results for $k$-rainbow domination functions of regular graphs. In Section~\ref{sec: dRDR}, we define a \emph{$d$-rainbow domination regular graph} as a $d$-regular graph such that $\gamma_{rk}(G)$ attains the lower bounds from Theorem~\ref{thm: gamma rk pred} for all $d\leq k\leq 2d$ and investigate necessary conditions for parameters of such graphs.
In Section~\ref{sec: rk cubic}, we investigate exact $k$-rainbow domination numbers of all connected cubic Cayley graphs over abelian groups for all $k$ and determine all $3$-rainbow domination regular graphs among these.
In Section~\ref{sec: remarks}, some further examples of $4$-rainbow domination regular graphs are given and some open questions are posed.

\section{Lower bounds for regular graphs}\label{sec: reg}

Unless otherwise noted, throughout this section, graph $G=(V,E)$ will be regular of order $n$ and degree $d$. 
By using elementary counting arguments, we shall obtain certain bounds on the weight of any $k$-rainbow domination function $f$ on $G$.

Suppose that graph $G=(V,E)$ and some function $f\colon V\to 2^{\cal C}$ are given. In what follows, we denote 
%the weight of given $f$ by $$w=w(f)=\sum_{v\in V}|f(v)|,$$ 
the non-disjoint sets of vertices that are colored with color $i\in{\cal C}$ 
by $$V_i=\{v\in V\colon i\in f(v)\} \text{ for }i=1,\ldots,k,$$
and by $V_0=\{v\in V\colon f(v)=\emptyset\}$ the set of non-colored vertices. Further, we denote 
the disjoint sets of vertices which are colored with exactly $i$ different colors by
$$C_i=\{v\in V\colon |f(v)|=i\} \text{ for }i=0,\ldots,k,$$
their union of all colored vertices by $$C=\bigcup_{i=1}^kC_i,$$
and the sets of edges with exactly $i$ end-vertices colored by
\begin{align*}
E_0&=\{\{u,v\}\in E\colon f(u)=f(v)=\emptyset\}, \\
E_1&=\{\{u,v\}\in E\colon f(u)=\emptyset\ne f(v)\}, \\
E_2&=\{\{u,v\}\in E\colon f(u)\ne\emptyset,f(v)\ne \emptyset\}.
\end{align*}
Also, we denote the sizes of respective sets by $$e_i=|E_i|,\quad n_i=|V_i|,\quad c_i=|C_i|,\quad
c=|C|=n-c_0.
$$ 

Our first observation is obtained by double counting the elements of $E_1$ in two different ways, and then omitting some terms to obtain upper and lower bounds for the number of $f$-colored vertices $c$.  
\begin{lemma}\label{lemma: <c<}
Let $G$ be a $d$-regular graph of order $n$ and let $f$ be some $k$-rainbow domination function on $G$. Then
\begin{align}e_1=cd-2e_2=c_0d-2e_0\label{e1=}\end{align}
and the following inequalities hold:
\begin{align}
c_0-\frac{2e_0}{d}&\leq c\leq c_0+\frac{2e_2}{d},\label{c<c0+}\\
\frac{n}{2}-\frac{e_0}{d}&\leq c\leq \frac{n}{2}+\frac{e_2}{d}.
\end{align}
\end{lemma}

Next, we use another double counting argument to get several lower bounds for the weight of a $k$-RDF.
\begin{lemma}\label{lemma: w(f)<}
Let $G$ be a $d$-regular graph of order $n$ and let $f$ be some $k$-rainbow domination function on $G$.
Then the following inequalities hold:
\begin{align}
w(f)&\geq \frac{(n-c)k+2e_2}{d},\label{w(f)1}\\
w(f)&\geq \frac{kn+2e_2}{k+d}\label{w(f)2},\\
w(f)&\geq \frac{(k-d)}{d}n+\frac{2d-k}{d}c,\label{w(f)3}\\
w(f)&\geq n-\left(\frac{2d-k}{d}\right)c_0.\label{w(f)4}
\end{align}
\end{lemma}
\begin{proof}
First, we count the number of ordered triples $(u,v,i)$, where $u\sim v$, $f(u)=\emptyset$ and $i\in f(v)$. Since each empty vertex is dominated by at least $k$ neighbours of different colors, we see this number is at least $(n-c)k$. On the other hand, there are exactly $w(f)$ pairs $(v,i)$ with $i\in f(v)$, and each of these has $d$ neighbours, some colored, others non-colored. Subtracting twice the number of edges with both edges colored we get the exact number of triples. Therefore,  $(n-c)k\leq w(f)d-2e_2$, and  inequality~(\ref{w(f)1}) follows.
Since $w(f)\geq c$, inequality (\ref{w(f)2}) is easily obtained from (\ref{w(f)1}). 
%Observe that with $n,k,d$ given, in order for $w(f)$ to be minimal,  number of edges with both vertices colored should be relatively small.
Eliminating $2e_2$ from inequalities (\ref{w(f)1}) and (\ref{c<c0+}), we combine them to obtain (\ref{w(f)3}) and then rewrite with $c=n-c_0$ to obtain (\ref{w(f)4}).
\end{proof}

We merge Lemmas~\ref{lemma: <c<} and \ref{lemma: w(f)<} to obtain the following Proposition.
%The inequalities of   Lemma~\ref{lemma: w(f)<} can be reinterpreted in the following way.
\begin{proposition}\label{thm: gamma rk}
Let $G$ be a  $d$-regular graph of order $n$ and let $f$ be some $k$-rainbow domination function on $G$. Then 
\begin{align*}
w(f)&\geq \left\lceil \frac{kn}{2d}\right\rceil \text{ for }k<2d,
\end{align*}
and $w(f)\geq n$ for $k\geq 2d.$
\end{proposition}

\begin{proof}
For $k<2d$, we check two cases.
If $c\geq n/2$, we apply inequality (\ref{w(f)3}) of Lemma~\ref{lemma: w(f)<} to get
$$w(f)\geq \frac{(k-d)}{d}n+\frac{2d-k}{d}\cdot\frac{n}{2}=\frac{kn}{2d}.$$
If $c<n/2$, then $n-c>n/2$ and we apply Lemma~\ref{lemma: <c<} to get $$w(f)> \frac{kn/2+2e_2}{d}\geq \frac{kn}{2d}.$$
In both cases, $w(f)$ is an integer, greater or equal to $\frac{kn}{2d}$.

For $k=2d$, we get $w(f)\geq n$ directly from (\ref{w(f)3}). For $k>2d$, we get $w(f)\geq 2(n-c)+\frac{2e_2}{d}$ from (\ref{w(f)1}) and hence $w(f)\geq n$ from (\ref{c<c0+}). 
\end{proof}

Since for any $\gamma_{rk}(G)$-function $f$ we have $w(f)=\gamma_{rk}(G)\leq n$, Theorem~\ref{thm: gamma rk pred} is a direct corollary of Proposition~\ref{thm: gamma rk}
and needs no further proof.

%\begin{remark}
%It is not difficult to see that the lower bound $\gamma_{rk}(G)\geq \left\lceil kn/2d\right\rceil$ from Theorem~\ref{thm: gamma rk} improves the lower bound $\gamma_{rk}(G)\geq\left\lceil kn/(k+\Delta)\right\rceil$ from Theorem \ref{thm: rk basic}  only for $2d\geq k>d$.
%\end{remark}

\begin{example}Exact $k$-rainbow domination numbers of cycles $C_n$, $n\geq 3$, were determined for $k=2$ by Bre\v sar and Kraner \cite{Bresar Kraner 2007} and for $k=3$ by Shao et al. \cite{Shao 2014} as
\begin{align}
\gamma_{r2}(C_n)&=\left\lceil \frac{n}{2}\right\rceil 
+\begin{cases} 1,&n\equiv 2\pmod 4,\\ 0,&\text{otherwise,}
\end{cases}\\
\gamma_{r3}(C_n)&=\left\lceil \frac{3n}{4}\right\rceil.
\end{align}
Using Theorem~\ref{thm: gamma rk pred}, we get $\gamma_{rk}(C_fn)=n$ for all $k\geq 4$ and lower bounds
$$
\gamma_{r2}(C_n)\geq \left\lceil \frac{n}{2}\right\rceil,
\gamma_{r3}(C_n)\geq \left\lceil \frac{3n}{4}\right\rceil,
$$
which can be used to shorten the original proofs significantly. 
%Moreover, it is easy to see that $\gamma_{r2}(C_n)=n/2$ and $\gamma_{r3}(C_n)=3n/4$ if and only if $4|n$. 
\end{example}

We can also rewrite inequality (\ref{w(f)3}) to obtain bounds for parameters $c$ and $c_0$.

\begin{proposition}
Let $G$ be a $d$-regular graph and let $f$ be a $\gamma_{rk}(G)$-function, where $0<k< 2d$. Then
\begin{align}
c&\leq \frac{d\gamma_{rk}(G)+(d-k)n}{2d-k}, \text{ and }\\
c_0&\geq \frac{d(n-\gamma_{rk}(G))}{2d-k}.\end{align}
\end{proposition}

\section{$d$-rainbow domination regular graphs}
\label{sec: dRDR}
The next theorem gives necessary conditions for a $d$-regular graph to attain the lower bound for $\gamma_{rk}(G)$.
\begin{theorem}\label{thm: =nk/2d}
Let $G$ be a $d$-regular graph of order $n$ and let $k<2d$. If $\gamma_{rk}(G)=\frac{kn}{2d}$, then $k\geq d$, $2d|n$ and $G$ is bipartite.

% with bipartition sets $C_0$ and $C$ of size $n/2$, and $n_i=n/2d$ for all $i=1,\ldots, k$, for some $\gamma_{rk}(G)$-function $f$.
\end{theorem}

% Èe zahtevam,da C1=C, dobim dodatne pogoje...
% where $C=V_1\cup\ldots \cup V_k$ is a disjoint union of subsets of size $\frac{n}{2k}$ 

% ... from the fixed proof:  This means each colored vertex is colored by a single color, and therefore $v_1=v_2=\ldots=v_k$. Thus, $kv_1=n/2$, so $k|\frac{n}{2}$. Also, $d|\frac{kn}{2}$.

\begin{proof}
For $k<2d$, let $w(f)=\frac{kn}{2d}$ for some $\gamma_{rk}(G)$-function $f$. 
Then $c<n/2$ is not possible by (\ref{w(f)1}) and
$c\geq n/2$ forces $c=n/2$ by (\ref{w(f)3}).  Since $\gamma_{rk}(G)\geq c$, we also get $k\geq d$. Moreover,  it follows from (\ref{w(f)1}) and (\ref{e1=})  that $e_0=e_2=0$ and $e_1=e$, so $G$ is bipartite with bipartition sets $C_0, C$ of size $n/2$. 
Now, count pairs $(i,u)$, such that $i\in{\cal C}$, $u\in V$, $f(u)=\emptyset$. Their number equals $v_i d$ and is at least $c_0=n/2$, so $v_i\geq n/2d$. But $w(f)=\sum i v_i\geq kn/2d$ implies equality $v_i=n/2d$ for all $i>0$, so $2d|n$.
(Observe that the proof also implies that for any $\gamma_{rk}(G)$-function $f$, equality $v_i=n/2d$ must hold for all colors $i\in{\cal C}$.)
\end{proof}

\begin{example}
The Franklin graph $F$ in Figure~\ref{fig: bicubic 12 not suff} shows that conditions of Theorem~\ref{thm: =nk/2d} are not sufficient. Since assigning a color $i\in{\cal C}$ to any two vertices in one bipartition set cannot dominate all 6 vertices in the other bipartition set, we get  $\gamma_{r3}(F)>6$. 

\begin{figure}[h!]\centering
\includegraphics[width=3cm]{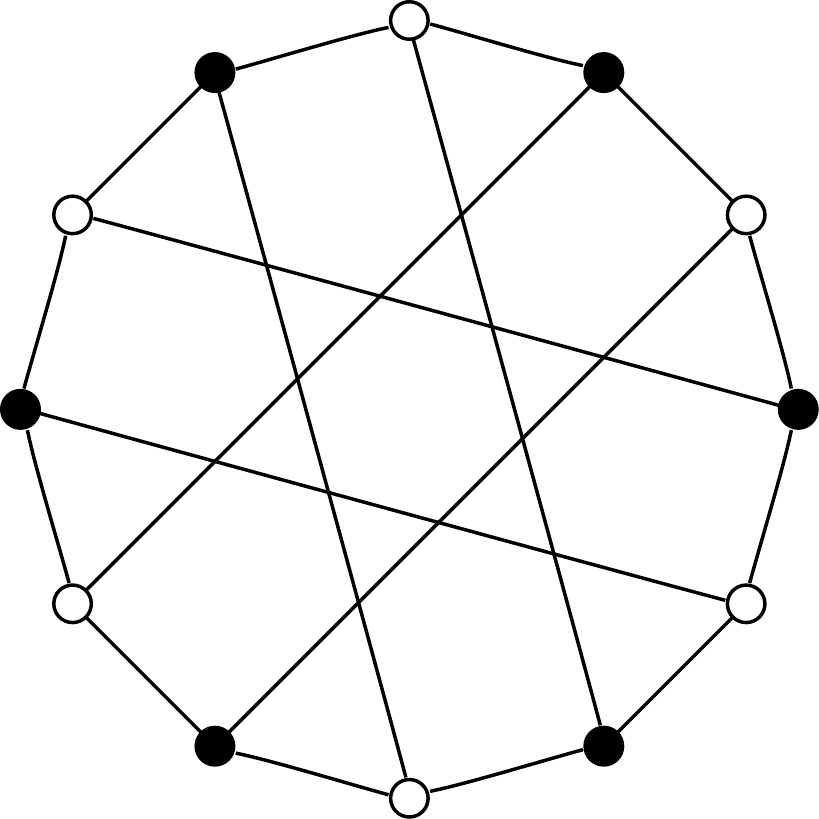}
\caption{Franklin graph is $3$-regular and bipartite with $\gamma_{r3}(F)> 3n/2d$.}
\label{fig: bicubic 12 not suff}
\end{figure}
\end{example}
\begin{example}
On the other hand, it is easy to see that the complete bipartite graph $G=K_{d,d}$ 
satisfies conditions of Theorem~\ref{thm: =nk/2d} and
%of order $n=2d$ and degree $d$ 
the equality $\gamma_{rk}(G)=\frac{kn}{2d}$ holds for all $d\leq k\leq 2d$. Indeed, for $d\leq k\leq 2d$, denote by $u_i, v_i\in V(G)$ the vertices in the bipartition sets and let
$$f(u_i)=\{j\in{\cal C}\colon j\equiv i\pmod d\} \text{ and }f(v_i)=\emptyset \>\text{ for }  i=1,\ldots,d.$$
Then each color appears exactly once and dominates $d$ non-colored vertices, so $f$ is a $k$-RDF of weight $w(f)=k=k\frac{n}{2d}$.

\end{example}
The `rainbow property' of graphs $K_{d,d}$ described above can be generalized as follows.
We define that graph $G$ is \emph{$d$-rainbow domination regular} (or $d$-RDR), if it is $d$-regular and equality $\gamma_{rd}(G)=\frac{|V(G)|}{2}$ holds. We now show this also implies a stronger condition.
\begin{proposition}
Suppose that $G$ is a $d$-regular graph of order $n$, and let $d\leq k <2d$ for some $k$.
If $\gamma_{rk}(G)=\frac{kn}{2d}$, then $\gamma_{r(k+1)}(G)=\frac{(k+1)n}{2d}$. 
\end{proposition}

\begin{proof}
Suppose that $d\leq k<2d$ and $f\colon V\to \{1,\ldots,k\}$ is $\gamma_{rk}(G)$-function with $\gamma_{rk}(G)=\frac{kn}{2d}$. Then $v_i=\frac{n}{2d}$ for all $i=1,\ldots,k$ by proof of Theorem~\ref{thm: =nk/2d} and a $\gamma_{r(k+1)}$-function $f'$ can be explicitly constructed from $f$ by selecting (any) color $i\in\{1,\ldots,k\}$ and
adding color $k+1$ to all $i$-colored vertices, that is, defining 
$$f'(v)=
\begin{cases} 
f(v)\cup  \{k+1\};&i\in f(v),\\
f(v);&i\notin f(v).\\
\end{cases}
$$
Obviously, $f'$ is a $(k+1)$-RDF of weight $w(f')=\sum_i v_i=\frac{(k+1)n}{2d}$.
\end{proof}
\begin{corollary}
Let $G$ be a $d$-regular graph of order $n$. Then $G$ is $d$-rainbow domination regular if and only if $\gamma_{rk}(G)=\frac{kn}{2d}$ for all $d\leq k\leq 2d$.
\end{corollary}
\begin{example}
For any $d\geq 1$, complete bipartite graph  $K_{d,d}$ is a $d$-RDR graph.
It is not difficult to see that a cycle $C_n$ is a $2$-RDR graph if and only if $n\equiv 0\pmod 4$.
\end{example}
We shall identify some more examples of $d$-RDR graphs in the next section.

\section{Cubic Cayley graphs over abelian groups}\label{sec: rk cubic}

In this section, we study the $k$-rainbow domination numbers for connected cubic (that is, $3$-regular) Cayley graphs over some finite abelian group $H$. Our motivation comes from the fact that Cayley graphs are a large class of regular graphs with nice symmetry properties that could provide some more examples of $d$-RDR graphs, and also from some existing results on values and bounds for $2$- and $3$-rainbow domination numbers of generalised Petersen graphs that partially overlap with our class and could be extended to larger $k$. 
For cubic graphs, we readily have 
$\gamma_{rk}(G)\geq  \frac{kn}{6}$ for $3 \geq k<6$
and $\gamma_{rk}(G)=n$ for $k\geq 6$ by Theorem~\ref{thm: gamma rk pred}.

\medskip

Recall that for any finite group $H$ and inverse closed subset $S=S^{-1}\subseteq H\setminus\{1_H\}$, the (non-directed and simple) \emph{Cayley graph} $G=\Cay(H, S)$ is defined by vertex set $V(G)=H$ and edge set $E(G)=\{\{g,h\}\colon g,h\in H, gh^{-1}\in S\}$. Graph $G=\Cay(H,S)$ is regular of degree $d=|S|$ and $G$ is connected if and only if $S$ is a generating set for group $H$. 

Note that in order for $\Cay(H,S)$ to be cubic, we must have either $S=\{a,a^{-1},b\}$, where $a\ne a^{-1}$ and $b=b^{-1}$, or $S=\{a,b,c\}$, where $a=a^{-1}$, $b=b^{-1}$ and $c=c^{-1}$. This implies the following well-known classification of such graphs. We omit details of the proof.
\begin{proposition}
Let $G=\Cay(H,S)$ be a connected cubic Cayley graph over some finite abelian group $H$. Then $G$ is isomorphic either to a $m$-sided prism $\Pr(m)$ or to a M\"obius ladder $\Ml(m)$, where $m=\frac{|H|}{2}\geq 2$ and graphs $\Pr(m)$ and $\Ml(m)$
of order $n=2m$ both have vertex set 
$V=\{u_i,v_i\colon i\in\ZZ_m\}$, while the edge set is equal to 
$E=\{\{u_i,v_i\},\{u_i,u_{i+1}\},\{v_i,v_{i+1}\}\colon i\in \ZZ_m\}$ for prisms $\Pr(m)$ and
to 
$E=\{\{u_i,v_i\},\{u_i,u_{i+1}\},\{v_i,v_{i+1}\}\colon i\in \ZZ_m, i\ne m-1\}\cup\{\{u_{m-1},v_{m-1}\},\{u_{m-1},v_0\},\{v_{m-1},u_0\}\}$ for M\"obius ladders $\Ml(m)$.
\end{proposition}
This reduces the problem to determination of $\gamma_{rk}(G)$ for prisms and M\"obius ladders. We shall state and prove separate theorems for each case.
Note that the prisms $\Pr(m)$ are just generalised Petersen graphs $GP(m,1)$, so for $k=2,3$, their $k$-rainbow domination numbers were already determined in \cite{Shao 2014, Shao 2019}. 
%For cubic graphs, we readily have 
%$\gamma_{rk}(G)\geq  \frac{kn}{6}$ for $3 \geq k<6$
%and $\gamma_{rk}(G)=n$ for $k\geq 6$ by Theorem~\ref{thm: gamma rk pred}.
%Moreover, in order for equality $\gamma_{rk}(G)=\frac{kn}{6}$ to hold, by Theorem \ref{thm: =nk/2d} we must have $6|n$, $3\leq k\leq 6$ and $G$ bipartite, which implies $m$ is even in case of prisms $\Pr(m)$ and $m$ is odd in case of M\"obius ladders $\Ml(m)$. 

\begin{theorem}\label{thm: prism}
Let $G=\Pr(m)$ be the $m$-sided prism of order $n=2m$, $m\geq 3$. Then
\begin{itemize}
\item $\gamma_{r1}(G)=\left\lceil\frac{m}{2}\right\rceil+
\begin{cases} 
1;& m \equiv 2\pmod 4,\\
0;& m\equiv 0,1,3\pmod 4.
\end{cases}$
\item $\gamma_{r2}(G)=m$.

\item $\gamma_{r3}(G)=m+
\begin{cases} 
0;& m \equiv 0\pmod 6,\\
1,& m\equiv 1,2,3,5\pmod 6,\\
2;& m\equiv 4\pmod 6.\\
\end{cases}$
\item
$\gamma_{r4}(G)=\left\lceil\frac{4m}{3}\right\rceil+
\begin{cases} 
0;& m \equiv 0,1\pmod 6,\\
1; &m\equiv 2,3,4,5\pmod 6,\\
\end{cases}
$

\item
$\gamma_{r5}(G)= \left \lceil\frac{5m}{3}\right\rceil+
\begin{cases} 
0;& m \equiv 0,1,2,5\pmod 6,\\
1; &m\equiv 3,4 \pmod 6.\\
\end{cases}
$

\item $\gamma_{rk}(G)=2m$ for $k\geq 6.$
\end{itemize}
Moreover, for $m\equiv 0\pmod 6$ and $3\leq k\leq 6$, the lower bound $\gamma_{rk}(G)=\frac{km}{3}$ is attained. For $k=4,5$, appropriate $\gamma_{rk}(G)$-functions
are given in Table~\ref{tbl: prizme}.
\end{theorem}
Before providing the proof, we state the following observation, which is valid for any graph $G$ (also non-regular) and is sometimes essential for finding exact values of $\gamma_{rk}(G)$.

\begin{lemma}[Discharging principle for $k$-rainbow domination]\label{discharging}
If a $\gamma_{rk}(G)$-function $f$ on $G$ has minimal number of non-colored vertices (that is, size $|\{v\colon f(v)=\emptyset\}|$ is minimal among all $\gamma_{rk}(G)$-functions on $G$), then 
$$|f(v)|\leq |\{u\colon u\sim v, f(u)=\emptyset\}|$$ for all $v\in V(G)$.
\end{lemma}
\begin{proof}
If the condition is not true for some $v\in V$, reassigning ("discharging") some colors from $v$ to its non-colored neighbours reduces $c_0$ without changing $w(f)$.
\end{proof}
%\begin{proof}
%Without loss of generality (=wlog), suppose $f(v)=\{1,2,\ldots, q\}$ and $q>|\{u\colon u\sim v,f(u)=\emptyset\}$. Then label thedefine $f'(u)=$
%\end{proof}

\begin{proof}[Proof of Theorem~\ref{thm: prism}]

For $k=1$, the result is well-known~\cite{wolfram domination number} and for $k=2,3$, the results are proven in \cite{Shao 2014, Shao 2019}. The case $k\geq 6$ is also clear. 

Let $k=4$.
Using Theorem~\ref{thm: gamma rk pred} for the lower bound and Theorem \ref{thm: t<t'} with $t=3=t'+1$ for the upper bound, we obtain 
\begin{equation}\label{bounds r4}
\left\lceil \frac{4m}{3}\right\rceil\leq 
\gamma_{r4}(G)
\leq \left\lceil\frac{4m}{3}\right\rceil+\begin{cases}
0;& m\equiv 0,1\pmod 6,\\
1;& m\equiv 2,3,5\pmod 6,\\ 
2;& m\equiv 4\pmod 6.\\
\end{cases}
\end{equation}
For $m=0,1\pmod 6$, we obviously have $\gamma_{r4}(G)=\left\lceil\frac{4n}{3}\right\rceil$.
For $m=3\pmod 6$, the prism $G=\Pr(m)$ is not bipartite, so we have 
a strict inequality $\left\lceil \frac{4m}{3}\right\rceil=\frac{4m}{3}<\gamma_{r4}(G)$ by Theorem~\ref{thm: =nk/2d}
, and 
$\gamma_{r4}(G)=\left\lceil\frac{4n}{3}\right\rceil+1$ follows.
It is not difficult to find appropriate $\gamma_{r4}(G)$-functions in these 3 cases, see Table~\ref{tbl: prizme}.

For $m\equiv 2,5\pmod 6$, suppose that 
$\gamma_{r4}(G)=\left\lceil\frac{4n}{3}\right\rceil$, and let $f$ be some $\gamma_{r4}(G)$-function. Using (\ref{w(f)4}) we obtain $c_0\geq m-1/2$, hence $c_0\geq m$. Since $m=6t+2$ or $6t+5$ and every color on a vertex dominates at most 3 empty vertices, this implies that $n_i\geq 2t+1$ or $n_i\geq 2t+2$, respectively, for $i=1,2,3,4$. But then $w(f)\geq 8t+4$ or $8t+8$, which is larger than $\gamma_{r4}(G)$, a contradiction. 
Hence  $\gamma_{r4}(G)=\left\lceil\frac{4n}{3}\right\rceil+1$ and appropriate $\gamma_{r4}(G)$-functions are easily constructed from the known $\gamma_{r3}(G)$-functions, see Table~\ref{tbl: prizme} and Remark~\ref{remark: construct f'}.

Finally, for $m\equiv 4\pmod 6$, we have $\gamma_{rk}(G)=\left\lceil\frac{4m}{3}\right\rceil+a$ with $a=0,1$ or $2$. Suppose first that $a=0$ and let $f$ be some $\gamma_{rk}(G)$-function. Then by (\ref{w(f)4}), we get $c_0\geq m-1$. If $c_0\geq m=6t+4$, we have $n_i\geq 2t+2$ for all $i$, so $w(f)\geq 8t+8>8t+6=\left\lceil 4m/3\right\rceil$. So $a=0$ implies $c_0=m-1$ and  $c=m+1$. From (\ref{w(f)1}) we further obtain $e_2=3$ and $e_0=0$, so every edge has at least one colored vertex.
This implies we may wlog (without loss of generality) suppose that the set of noncolored vertices is $C_0=\{u_0,u_2,\ldots ,u_{m-2}, v_1,v_3,\ldots,v_{m-3}\}$ (note that $f(u_{m-1}),f(v_{m-1})\notin C_0$). As in the proof of Theorem 4 in \cite{Shao 2014}, we now consider subsets $\P_i=\{v_i,v_{i+1},v_{i-1},u_i,u_{i+1},u_{i-1}\}$
for $i\in\ZZ_m$, and denote by $\gamma_i=\sum_{v\in\P_i}|f(v)|$ the $f$-weight of $\P_i$. 
Observe that $\gamma_0\geq 5$, since $u_0$ is dominated by $4$ colors and $v_{m-1}\ne \emptyset$. Similarily, $\gamma_{m-2}\geq 5$, while $\gamma_i\geq 4$ for all other $i$.
Since each vertex belongs to exactly three blocks $\P_i$, we see that 
$$4m+2=3 w(f)=\sum_i \gamma_i\geq  4(m-2)+5\cdot 2.$$
Hence, equality holds and we have $\gamma_i=4$ for $i\ne 0,m-2$.
In particular, $\gamma_{m-1}=4$ implies $|f(u_{m-1})|\leq 1$ and $|f(v_0)|\leq 1$, and $\gamma_0=5$ further implies $|f(u_1)|=2$. Continuing, we see that $|f(v_4)|=|f(u_7)|=\ldots=|f(u_{m-3})|=2.$, while $|f(v)|=1$ for all other vertices from $C$. But now, it is easy to see that any choice of colors forces $f(u_{m-1})=f(v_{m-2})$, hence $u_{m-2}$ is not dominated by all colors. This contradiction shows $a\geq 1$, but in fact, $a=1$, as we can construct an appropriate $4$-RD function of weight $\left\lceil\frac{4m}{3}\right\rceil+1$, see Table~\ref{tbl: prizme}.

\medskip

Now, let $k=5$. Using the known bounds and values for $\gamma_{r4}(G)$, with some computation we obtain 
$$
\left\lceil\frac{5m}{3}\right\rceil \leq
\gamma_{r5}(G)\leq 
\left\lceil\frac{5m}{3}\right\rceil +
\begin{cases}
0;& m\equiv 0,1\pmod 6,\\
1;& m\equiv 2,3,4,5\pmod 6.
\end{cases}
$$
As in case $k=4$, the values of $\gamma_{rk}(G)$ are obtained trivially for $m\equiv 0,1,3\pmod 6$.

For $m\equiv 2,4,5\pmod 6$, write $m=6t+2,4,5$, resp., and suppose that $\gamma_{rk}(G)=\left\lceil 5n/3\right\rceil=10t+4,7,9$, respectively. Let $f$ be a $\gamma_{rk}(G)$-function. It follows from (\ref{w(f)4}) that $c_0\geq m-a$, with $a=2,1,2$ for $m\equiv 2,4,5\pmod 6$, respectively.  If $c_0\geq m$, we get $n_i\geq 2t+1,2,2$, resp., hence $w(f)\geq 5n_i=10t+5,10,10$, a contradiction. Similarly, we get that $c=m-1$ is not possible for $m=2,5\pmod 6$, while for $m=4\pmod 6$ we apply (\ref{w(f)1}) to get $e_2\leq 3$. Moreover, it now follows from (\ref{e1=}) that $e_2-e_0=3$, implying that $e_2=3$ and $e_0=0$. As in case $k=4$, we now wlog suppose that $C_0=\{u_0,u_2,\ldots,u_{m-2},v_1,v_3,\ldots,v_{m-1}\}$ and observe sets $\P_i$ with weights $\gamma_i$ to get $\gamma_0,\gamma_{m-2}\geq 6$ and $\gamma_i\geq 4$ otherwise. From $3w(f)=\sum_i\gamma_i$ we get that the equalities hold, implying that $|f(u_1)|=|f(v_4)|=\ldots=|f(u_{m-3})|=3$ and $|f(v)|=1$ for other $v\notin C_0$, implying further that $f(v_0)=f(u_{m-1})$, a contradiction. Finally, suppose $c=m-2$ with $m=6t+2$ or $5$. Repeating above arguments, we get that this case is possible with $e_2=6$, $e_0=0$. Inspecting possible sets $\P_i$ and weights $\gamma_i$, it is now not difficult to construct appropriate $\gamma_{rk}(G)$-functions for both cases.
\end{proof}

\renewcommand{\arraystretch}{1.3}
\begin{table}\centering
\begin{tabular}{|c|c|l|c|l|}
\hline
$s$&$\gamma_{r4}(G)$&
$f\left(\begin{smallmatrix}
u_0&u_1&u_2&u_3&u_4&u_5&\cdots & u_{m-1}\\
v_0&v_1&v_2&v_3&v_4&v_5&\cdots & v_{m-1}
\end{smallmatrix}\right)$
&$\gamma_{r5}(G)$
&
$f\left(\begin{smallmatrix}
u_0&u_1&u_2&u_3&u_4&u_5&\cdots & u_{m-1}\\
v_0&v_1&v_2&v_3&v_4&v_5&\cdots & v_{m-1}
\end{smallmatrix}\right)$
\\ \hline

$0$&
$\frac{4}{3}m$&
%$8t$&
$\begin{smallmatrix}
\emptyset& 34& \emptyset& 1&  \emptyset&  2&\ldots& \\
1&  \emptyset&  2& \emptyset& 34& \emptyset &\ldots&
\end{smallmatrix}$
&
$\frac{5}{3}m$&
$\begin{smallmatrix}
\emptyset& 345& \emptyset& 1&  \emptyset&  2&\ldots& \\
1&  \emptyset&  2& \emptyset& 345& \emptyset &\ldots&
\end{smallmatrix}$
\\ \hline

$1$&
%$8t+2$&
$\frac{4}{3}m+\frac{2}{3}$&
$\begin{smallmatrix}
\emptyset& 34& \emptyset& 1&  \emptyset&  2&\ldots& 2\\
1&  \emptyset&  2& \emptyset& 34& \emptyset &\ldots&1
\end{smallmatrix}$
&
$\frac{5}{3}m+\frac{1}{3}$&
$\begin{smallmatrix}
\emptyset& 345& \emptyset& 1&  \emptyset&  2&\ldots& 2\\
1&  \emptyset&  2& \emptyset& 345& \emptyset &\ldots&1
\end{smallmatrix}$
\\ \hline

$2$&
$\frac{4}{3}m+\frac{4}{3}$&
$\begin{smallmatrix}
\emptyset& 1& \emptyset   & 34&  \emptyset&  2&\ldots& \emptyset&2\\
34&  \emptyset&  2& \emptyset& 1& \emptyset &\ldots&134&\emptyset
\end{smallmatrix}$
&
$\frac{5}{3}m+\frac{2}{3}$&
$\begin{smallmatrix}
\emptyset& 345& \emptyset   & 1&  \emptyset&  2&\ldots&3&2\\
1&  \emptyset&  2& \emptyset& 345& \emptyset &\ldots&1&4
\end{smallmatrix}$

\\ \hline

$3$&
$\frac{4}{3}m+1$&
$\begin{smallmatrix}
\emptyset& 34& \emptyset   &2&  \emptyset&  1&\ldots& \emptyset&34&1\\
2&  \emptyset&  1& \emptyset& 34& \emptyset &\ldots&2&\emptyset&1
\end{smallmatrix}$
&
$\frac{5}{3}m+1$&
$\begin{smallmatrix}
\emptyset& 345& \emptyset   &1&  \emptyset&  2&\ldots& \emptyset&345&2\\
1&  \emptyset&  2& \emptyset& 345& \emptyset &\ldots&2&\emptyset&2
\end{smallmatrix}$

\\ \hline

$4$&
$\frac{4}{3}m+\frac{5}{3}$&
$\begin{smallmatrix}
\emptyset&2&\emptyset&1&\emptyset&34&\cdots&\emptyset&12&\emptyset&34\\
1&\emptyset&34&\emptyset&2&\emptyset&\cdots&1&3&2&\emptyset
\end{smallmatrix}$
&
$\frac{5}{3}m+\frac{4}{3}$&
$\begin{smallmatrix}
\emptyset&2&\emptyset&1&\emptyset&345&\cdots&\emptyset&12&\emptyset&345\\
1&\emptyset&345&\emptyset&2&\emptyset&\cdots&1&3&2&\emptyset
\end{smallmatrix}$

\\ \hline

$5$&
$\frac{4}{3}m+\frac{4}{3}$&
$\begin{smallmatrix}
\emptyset&34&\emptyset&1&\emptyset&2&\cdots&\emptyset&34&\emptyset&1&2\\
1&\emptyset&2&\emptyset&34&\emptyset&\cdots&1&\emptyset&2&\emptyset&34
\end{smallmatrix}$
&
$\frac{5}{3}m+\frac{2}{3}$&
$\begin{smallmatrix}
\emptyset&2&\emptyset&1&\emptyset&345&\cdots&\emptyset&2&1&\emptyset&345\\
1&\emptyset&345&\emptyset&2&\emptyset&\cdots&1&3&4&2&\emptyset
\end{smallmatrix}$
\\ 
\hline
\end{tabular}
\caption{$\gamma_{rk}(G)$-functions of $G=\Pr(m)$ for $k=4,5$, where
$m=6t+s\geq 3$.}
\label{tbl: prizme}

\end{table}

\begin{remark}~\label{remark: construct f'}
It follows from the proof of Theorem 1 in \cite{Shao 2014}, that whenever equality $\gamma_{r(k+1)}(G)=\gamma_{rk}(G)+\left\lfloor \gamma_{rk}(G)/k\right\rfloor$ holds
and a $\gamma_{rk}(G)$-function is known, a $\gamma_{r(k+1)}$-function $f'$ can be explicitly constructed from $f$ by selecting color $i\in{\cal C}=\{1,\ldots,k\}$ such that $n_i=\min\{n_1,\ldots,n_k\}$ and 
adding color $k+1$ to all $i$-colored vertices, that is, defining 
$$f'(v)=
\begin{cases} 
f(v)\cup  \{k+1\};&i\in f(v),\\
f(v);&i\notin f(v).\\
\end{cases}
$$
\end{remark}

For M\"obius ladders, we obtain very similar results.
\begin{theorem}\label{thm: mobius}
Let $G=\Ml(m)$ be the M\"obius ladder of order $n=2m$, $m\geq 2$. Then
\begin{itemize}
\item $\gamma_{r1}(G)=\left\lceil\frac{m}{2}\right\rceil+
\begin{cases} 
1;& m \equiv 0\pmod 4,\\
0;& m\equiv 1,2,3\pmod 4.
\end{cases}$
\item $\gamma_{r2}(G)=m$.

\item $\gamma_{r3}(G)=m+
\begin{cases} 
0;& m \equiv 3\pmod 6,\\
1,& m\equiv 0,2,4,5\pmod 6,\\
2;& m\equiv 1\pmod 6.\\
\end{cases}$
\item
$\gamma_{r4}(G)=\left\lceil\frac{4m}{3}\right\rceil+
\begin{cases} 
0;& m \equiv 3,4\pmod 6,\\
1; &m\equiv 0,1,2,5\pmod 6,\\
\end{cases}
$

\item
$\gamma_{r5}(G)= \left \lceil\frac{5m}{3}\right\rceil+
\begin{cases} 
0;& m \equiv 2,3,4,5 \pmod 6,\\
1; &m\equiv 0,1 \pmod 6.\\
\end{cases}
$

\item $\gamma_{rk}(G)=2m$ for $k\geq 6.$
\end{itemize}
Moreover, for $m\equiv 3\pmod 6$ and $3\leq k\leq 6$, the lower bound $\gamma_{rk}(G)=\frac{km}{3}$ is attained. For $k=3,4,5$, appropriate $\gamma_{rk}(G)$-functions
are given in Table~\ref{tbl: lojtre}.
\end{theorem}

\begin{proof}
For $k=1$, the result is well known, see~\cite{wolfram domination number}. 
For $k=2$, we can define a $2$-RDF $f$ of weight $w(f)=m$ on $G$ by setting $f(u_{2i})=\{1\}$, $f(v_{2i+1})=\{2\}$ and $f(v)=\emptyset$ otherwise, so $\gamma_{r2}(G)\leq m$. Now suppose $\gamma_{r2}(G)\leq m-1$ and let $f$ be an appropriate $\gamma_{r2}(G)$-function of weight $w(f)\leq m-1$. Note that we can wlog suppose that $f$ has the minimal number of non-colored vertices and apply Lemma~\ref{discharging} (Discharging principle) whenever needed. Denote by 
$\P_i=\{u_i,v_i,u_{i+1},v_{i+1},u_{i-1},v_{i-1}\}$ and $\gamma_i=\sum_{v\in \P_i}|f(v)|$ for $i\in\ZZ_m$. Observe first that $\gamma_i\geq 0,1$ is not possible, so $\gamma_i\geq 2$.
Now denote by $\alpha$ the number of blocks $\P_i$ with $\gamma_i=2$ and by $\beta$ the number of blocks with $\gamma_i\geq 4$. Then we have
$$3(m-1)\geq 3w(f)=\sum_i\gamma_i\geq 2\alpha+4\beta+3(m-\alpha-\beta),$$
hence $\alpha\geq 3+\beta$.  Now let $\gamma_i=2$ and inspect the possible  values of $f$ on $\P_i$.
\begin{itemize}
\item It is easy to see that blocks with values
$\left(\begin{smallmatrix}
f(u_{i-1})&f(u_i)&f(u_{i+1})\\
f(v_{i-1})&f(v_i)&f(v_{i+1})
\end{smallmatrix}\right)$
equal to 
$\left(\begin{smallmatrix}
12&0&0\\0&0&0
\end{smallmatrix}\right),$
$\left(\begin{smallmatrix}
1&0&0\\2&0&0
\end{smallmatrix}\right),$
$\left(\begin{smallmatrix}
1&0&0\\1&0&0
\end{smallmatrix}\right),$
$\left(\begin{smallmatrix}
1&0&0\\0&0&2
\end{smallmatrix}\right),$
$\left(\begin{smallmatrix}
1&0&0\\0&0&1
\end{smallmatrix}\right)$,
$\left(\begin{smallmatrix}
1&0&0\\0&1&0
\end{smallmatrix}\right)$
or reflections of such blocks
are not possible.
\item The blocks of types
$\left(\begin{smallmatrix}
0&12&0\\0&0&0
\end{smallmatrix}\right),$
$\left(\begin{smallmatrix}
0&1&0\\0&2&0
\end{smallmatrix}\right),$
$\left(\begin{smallmatrix}
0&1&0\\0&1&0
\end{smallmatrix}\right),$
are possible, but for each of them we have $\gamma_{i\pm1}\geq 4$, so each their occurence also increases $\beta$. Since $\alpha\geq \beta+3$, we may wlog assume that there are no such blocks.
\item Thus, we have at least $\alpha$ blocks of type
$\left(\begin{smallmatrix}
0&1&0\\2&0&0
\end{smallmatrix}\right)$ or
$\left(\begin{smallmatrix}
0&1&0\\0&0&2
\end{smallmatrix}\right).$
However, for each such block $\P_i$ either $\gamma_{i+1}$ or $\gamma_{i-1}$ is at least $4$.
If $\gamma_i=\gamma_{i+1}=2$, then $\gamma_{i-1}=4$ and $\gamma_{i+2}=4$, hence the average weight of these four blocks is $3$. If $\gamma_i=2$ and $\gamma_{i+1}=3$, then we must have $\gamma_{i-1}=4$ and $\gamma_{i-2}=3$, and their average weight is again $3$. Hence, removing any such quadruplet of blocks from the equation still forces $\alpha\geq \beta+3$, but now there are no more possibilities for $\P_i$ with $\gamma_i=2$, a contradiction.
\end{itemize}

Let $k=3$. Then $\left\lceil nk/2d\right\rceil=m$ is the lower bound for $\gamma_{r3}(G)$ for all $m$. However, for the lower bound to be exact we must have $6|2m$ and $G$ bipartite, which only happens for $m=3\pmod 6$. For $m=0,1,4,5\pmod 6$, we have $\gamma_{r3}(G)\geq m+1$, but in fact equality holds, as appropriate functions are easily constructed, see Table~\ref{tbl: lojtre 3}. 

It remains to show that $\gamma_{r3}= m+2$ for $m=6t+1$. Again, by Table~\ref{tbl: lojtre 3} we confirm that this is the upper bound.
Now suppose $w(f)=m+1=6t+2$ for some $3$-RDF $f$. Then $c\leq w(f)$ and hence $c_0\geq 6t$. Since each colored vertex dominates at most $3$ non-colored vertices, we must have $n_1,n_2,n_3\geq 2t$, but also $n_1+n_2+n_3=6t+2$, so $n_i=2t$ for at least one $i$.
Wlog assume $n_3=2t$. Then $c_0\leq 6t$ and hence $c_0=6t$ and $c=6t+2=w(f)$, which implies $c_1=c$ and $c_2=c_3=0$, so each vertex is colored with at most $1$ color.
By checking all possible cases for adjacent pairs of vertices, that is, $u_i\sim v_i$, $u_i\sim u_{i+1}$, or $u_{m-1}\sim v_0$, it is now easy to see that $f(u)=f(v)=\emptyset$ for some pair $u\sim v$ implies $|f(w)|=2$ for some $w$ adjacent to $u$ or $v$, a contradiction. Thus, we have $e_0=0$ and therefore $e_2=3$, so we have exactly 3 edges with both end-vertices colored (and hence no 4-cycle with all vertices colored). Note also, that since $n_3=2t$, color $3$ cannot be used on such edges.

We now check all different pairs of $u\sim v$ with $|f(u)|=|f(v)|=1$ to arrive at the contradiction. (Alternatively, we could denote $\P_i=\{u_i,v_i,u_{i+1},v_{i+1},u_{i-1},v_{i-1}\}$ and $\gamma_i=\sum_{v\in \P_i}|f(v)|$ to see that $3\leq \gamma_i\leq 4$ and $\sum_i \gamma_i=3w(f)$ implies that $\gamma_i=4$ for exactly $3$ consecutive $i$, and get the contradiction after some further inspection.)
First, we check pairs $u_i\sim v_i$ with $i=1,\ldots, m-2$. Wlog $f(u_1)=f(v_1)\ne 0$. Then exactly one of $u_0,v_0$ is non-colored, say $f(u_0)=\emptyset, f(v_0)\ne \emptyset$. Also, exactly one of $u_2,v_2$ is non-colored:
\begin{itemize}
\item If $f(u_2)=\emptyset$, we have $f(v_2)\ne \emptyset$ and so $v_0v_1, v_1u_1, v_2v_3$ are the 3 edges with both end-vertices colored, therefore exactly one of $u_i,v_i$ is noncolored for all other $i$. This implies $f(v_3)=f(u_4)=f(v_5)=\ldots=f(v_{n-2})=f(u_{n-1})=\emptyset$.
Since $f(v)=\{3\}$ only for vertices with 3 non-colored neighbours, we have $\{3\}=f(u_3)=f(v_6)=\ldots =f(v_{m-1})$. Now take wlog $f(v_2)=\{2\}$ and $f(u_1)=\{1\}$ and see this forces $\{2\}=f(v_2)=f(u_5)=\ldots=f(u_1)$, a contradiction.
\item If $f(u_2)\ne\emptyset$, we have $f(v_2)=\emptyset$ and so $v_0v_1,v_1u_1,u_1u_2$ are the 3 edges with both end-vertices colored. Again exactly one of $u_i,v_i$ is noncolored for all other $i$ implying that $f(v_{m-1})=\emptyset$, a contradiction, since $v_{m-1}\sim u_0$ and $f(u_0)=\emptyset$.
so for all other $j$, we have exactly one of $u_j,v_j$ colored. 
\end{itemize}
In similar fashion, we obtain a contradiction in all other cases. Therefore $\gamma_{r3}(G)=m+2$ for $m\equiv 1\pmod 6$.

For $k=4,5$, the proof is similar as the proof for prisms. For instance, we use Theorems~\ref{thm: t<t'} and \ref{thm: gamma rk pred} to obtain 
$$\left\lceil \frac{4}{3}m\right\rceil \leq \gamma_{r4}(G)\leq 
\left\lceil \frac{4}{3}m\right\rceil+\begin{cases}
0,&m\equiv 3,4\pmod 6;\\
1,&m\equiv 0,2,5\pmod 6;\\
2,&m\equiv 1\pmod 6.
\end{cases}$$ 
Then we find an appropriate $4$-RDF (see Table~\ref{tbl: lojtre}) and apply different combinatorial arguments to prove that lower weight is not possible. We omit further details.
\end{proof}

\begin{table}\label{tbl: lojtre 3}
\centering
\begin{tabular}{|c|c|l|c|l|}
\hline
$s$&$\gamma_{r3}(G)$&
$f\left(\begin{smallmatrix}
u_0&u_1&u_2&u_3&u_4&u_5&\cdots & u_{m-1}\\
v_0&v_1&v_2&v_3&v_4&v_5&\cdots & v_{m-1}
\end{smallmatrix}\right)$
\\ \hline
$0$&
$m+1$&
$\begin{smallmatrix}
\emptyset& 2& \emptyset& 1&  \emptyset&  3&\ldots& \emptyset& 2& \emptyset& 1&  \emptyset&3  \\
1&  \emptyset&  3& \emptyset& 2& \emptyset &\ldots&
1&  \emptyset&  3& \emptyset& 2& 3
\end{smallmatrix}$
\\ \hline

$1$&
$m+2$&
$\begin{smallmatrix}
\emptyset& 2& \emptyset& 1&  \emptyset&  3&\ldots& 1\\
1&  \emptyset&  3& \emptyset& 2& \emptyset &\ldots&1
\end{smallmatrix}$
\\ \hline

$2$&
$m+1$&
$\begin{smallmatrix}
\emptyset& 2& \emptyset& 1&  \emptyset&  3&\ldots& 0&2\\
1&  \emptyset&  3& \emptyset& 2& \emptyset &\ldots&1&3
\end{smallmatrix}$
\\ \hline

$3$&
$m$&
$\begin{smallmatrix}
\emptyset& 2& \emptyset& 1&  \emptyset&  3&\ldots& 0&2&0\\
1&  \emptyset&  3& \emptyset& 2& \emptyset &\ldots&1&0&3
\end{smallmatrix}$
\\ \hline

$4$&
$m+1$&
$\begin{smallmatrix}
\emptyset& 2& \emptyset& 1&  \emptyset&  3&\ldots& 0&2&0&1\\
1&  \emptyset&  3& \emptyset& 2& \emptyset &\ldots&1&3&3&3
\end{smallmatrix}$
\\ \hline

$5$&
$m+1$&
$\begin{smallmatrix}
\emptyset& 2& \emptyset& 1&  \emptyset&  3&\ldots& 0&2&0&1&3\\
1&  \emptyset&  3& \emptyset& 2& \emptyset &\ldots&1&0&3&0&2
\end{smallmatrix}$
\\ 
\hline
\end{tabular}

\caption{$\gamma_{r3}(G)$-functions for $G=\Ml(m)$, where $m=6t+s\geq 2$.}
\end{table}
%%%%%% Mobiusove lojtre
\begin{table}\centering
\begin{tabular}{|c|c|l|c|l|}
\hline
$s$&$\gamma_{r4}(G)$&
$f\left(\begin{smallmatrix}
u_0&u_1&u_2&u_3&u_4&u_5&\cdots & u_{m-1}\\
v_0&v_1&v_2&v_3&v_4&v_5&\cdots & v_{m-1}
\end{smallmatrix}\right)$
&$\gamma_{r5}(G)$
&
$f\left(\begin{smallmatrix}
u_0&u_1&u_2&u_3&u_4&u_5&\cdots & u_{m-1}\\
v_0&v_1&v_2&v_3&v_4&v_5&\cdots & v_{m-1}
\end{smallmatrix}\right)$
\\ \hline
%& & 
%$\begin{smallmatrix}
%u_0&u_1&u_2&u_3&u_4&u_5&\cdots & u_{m-1}\\
%v_0&v_1&v_2&v_3&v_4&v_5&\cdots & v_{m-1}
%\end{smallmatrix}$
%&&
%$\begin{smallmatrix}
%u_0&u_1&u_2&u_3&u_4&u_5&\cdots & u_{m-1}\\
%v_0&v_1&v_2&v_3&v_4&v_5&\cdots & v_{m-1}
%\end{smallmatrix}$
%\\
$0$&
$\frac{4}{3}m+1$&
$\begin{smallmatrix}
\emptyset& 34& \emptyset& 1&  \emptyset&  2&\ldots& \emptyset& 34& \emptyset& 1&  \emptyset&  2&\\
1&  \emptyset&  2& \emptyset& 34& \emptyset &\ldots&
1&  \emptyset&  2& \emptyset& 34& 2 
\end{smallmatrix}$
&
$\frac{5}{3}m+1$&
$\begin{smallmatrix}
\emptyset&345&\emptyset&1&\emptyset&2&\cdots&\emptyset&345&\emptyset
&1&\emptyset&2\\
1&\emptyset&2&\emptyset&345&\emptyset&\cdots&1&\emptyset&2
&\emptyset&345&2
\end{smallmatrix}$
\\ \hline

$1$&
$\frac{4}{3}m+\frac{5}{3}$&
$\begin{smallmatrix}
\emptyset& 34& \emptyset& 1&  \emptyset&  2&\ldots& 3\\
1&  \emptyset&  2& \emptyset& 34& \emptyset &\ldots&12
\end{smallmatrix}$
&
$\frac{5}{3}m+\frac{4}{3}$&
$\begin{smallmatrix}
\emptyset&345&\emptyset&1&\emptyset&2&\cdots&3\\
1&\emptyset&2&\emptyset&345&\emptyset&\cdots&12
\end{smallmatrix}$
\\ \hline

$2$&
$\frac{4}{3}m+\frac{4}{3}$&
$\begin{smallmatrix}
\emptyset& 34& \emptyset& 1&  \emptyset&  2&\ldots& 0&34\\
1&  \emptyset&  2& \emptyset& 34& \emptyset &\ldots&1&2
\end{smallmatrix}$
&
$\frac{5}{3}m+\frac{2}{3}$&
$\begin{smallmatrix}
\emptyset&345&\emptyset&1&\emptyset&2&\cdots&3&4\\
1&\emptyset&2&\emptyset&345&\emptyset&\cdots&1&2
\end{smallmatrix}$

\\ \hline

$3$&
$\frac{4}{3}m$&
$\begin{smallmatrix}
\emptyset& 34& \emptyset   &1&  \emptyset&  2&\ldots& \emptyset&34&\emptyset\\
1&  \emptyset&  2& \emptyset& 34& \emptyset &\ldots&1&\emptyset&2
\end{smallmatrix}$
&
$\frac{5}{3}m$&
$\begin{smallmatrix}
\emptyset&345&\emptyset&1&\emptyset&2&\cdots&\emptyset&345&\emptyset\\
1&\emptyset&2&\emptyset&345&\emptyset&\cdots&1&\emptyset&2
\end{smallmatrix}$

\\ \hline

$4$&
$\frac{4}{3}m+\frac{2}{3}$&
$\begin{smallmatrix}
\emptyset& 34& \emptyset   &1&  \emptyset&  2&\ldots& \emptyset&34&0&1\\
1&  \emptyset&  2& \emptyset& 34& \emptyset &\ldots&1&0&2&2
\end{smallmatrix}$
&
$\frac{5}{3}m+\frac{1}{3}$&
$\begin{smallmatrix}
\emptyset&345&\emptyset&1&\emptyset&2&\cdots&\emptyset&345&\emptyset&1\\
1&\emptyset&2&\emptyset&345&\emptyset&\cdots&1&\emptyset&2&2
\end{smallmatrix}$

\\ \hline

$5$&
$\frac{4}{3}m+\frac{4}{3}$&
$\begin{smallmatrix}
\emptyset& 34& \emptyset   &1&  \emptyset&  2&\ldots& \emptyset&34&\emptyset&1&4\\
1&  \emptyset&  2& \emptyset& 34& \emptyset &\ldots&1&\emptyset&2&3&2
\end{smallmatrix}$
&
$\frac{5}{3}m+\frac{2}{3}$&
$\begin{smallmatrix}
\emptyset&345&\emptyset&1&\emptyset&2&\cdots&\emptyset&345&\emptyset&1&4\\
1&\emptyset&2&\emptyset&345&\emptyset&\cdots&1&\emptyset&2&3&2
\end{smallmatrix}$
\\ 
\hline
\end{tabular}
\caption{$\gamma_{rk}(G)$-functions of $G=\Ml(m)$ for $k=4,5$ where $m=6t+s\geq2$.}
\label{tbl: lojtre}

\end{table}
%From results above, we can now extract.
\begin{corollary}
The only connected $3$-rainbow domination regular Cayley graphs over abelian groups are prisms $\Pr(m)$, where $m\equiv 0\pmod 6$, and M\"obius ladders $\Ml(m)$, where $m\equiv 3\pmod 6$.
\end{corollary}

\section{Concluding remarks}\label{sec: remarks}

The investigations of $k$-rainbow domination numbers for $d$-regular graphs from previous sections can be naturally generalized to some other classes of cubic graphs or to specific classes of $d$-regular graphs with $d\geq 4$. However, determining exact $k$-rainbow domination numbers for all $4$-regular Cayley graphs over abelian groups might turn out to be quite difficult already.

It would also be interesting to obtain further classification of $d$-RDR graphs. For two more examples, see Figure~\ref{4rd}. It is easy to check that the
tesseract graph $Q_4$ is a $4$-RDR graph, and the wreath graph $G=C_m[2K_1]$ (the lexicographic product of a cycle with 2 isolated vertices) is a $4$-RDR graph for all $m\geq 3$ with $m\equiv 0\pmod 4$.

\begin{figure}[h!]\label{4rd}
\centering
\includegraphics[width=6cm]{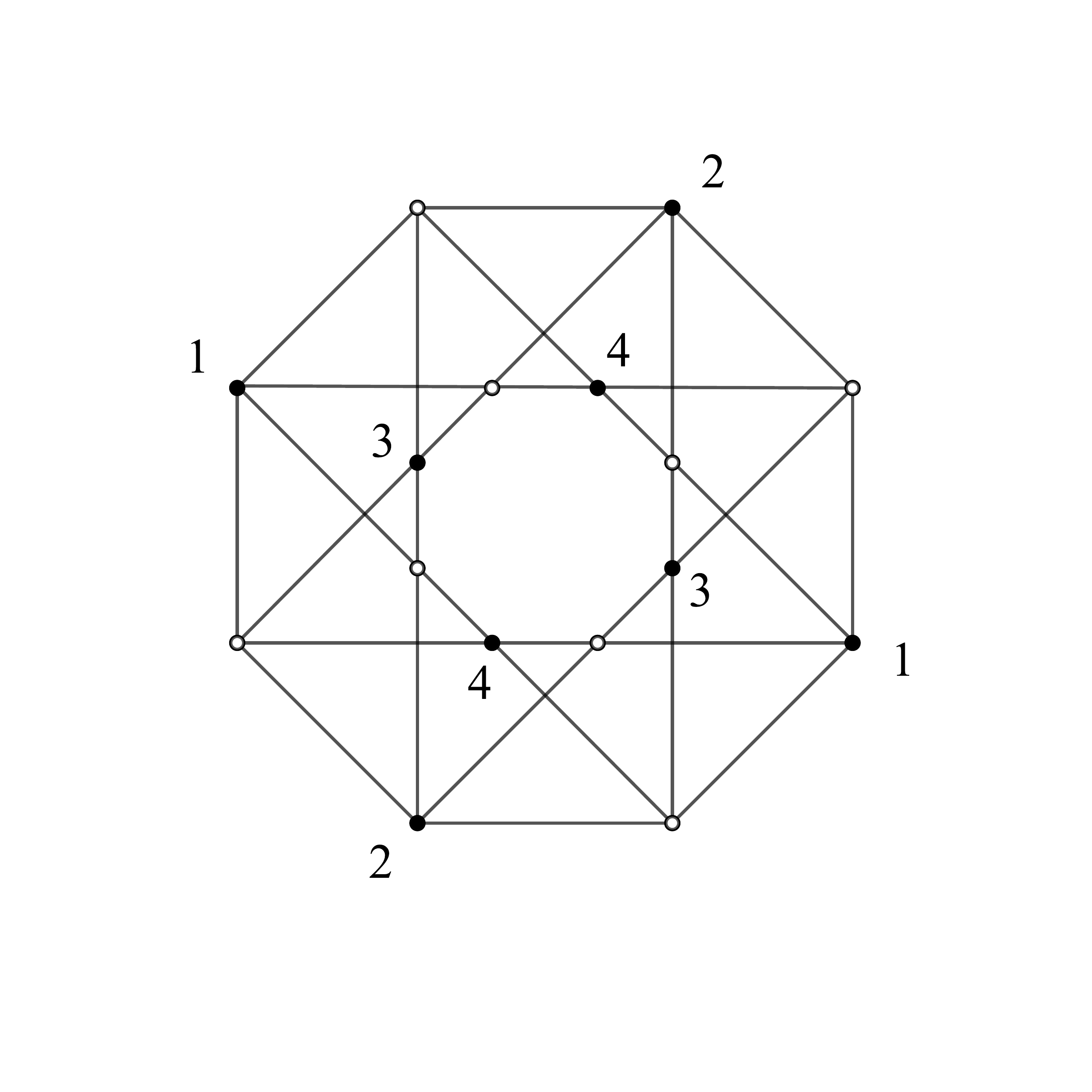}
\includegraphics[width=6cm]{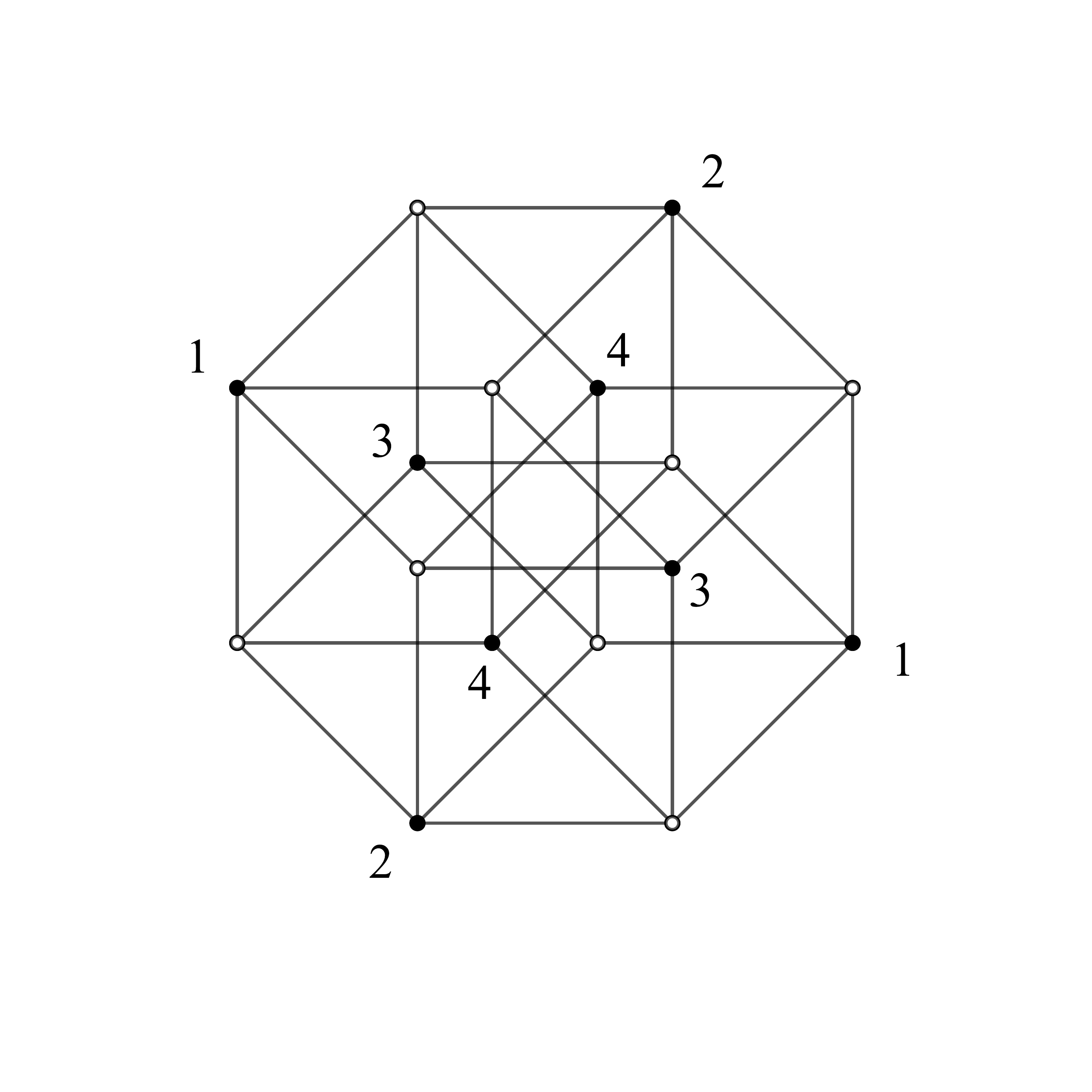}
\caption{Two $4$-RDR graphs, the wreath graph $C_4[2K_1]$ and the tesseract $Q_4$, with indicated values of their $\gamma_{r4}$-functions.}
\end{figure}

%As all other examples of $d$-RDR graphs we have seen thus far, these two graphs also arise as Cayley graphs over abelian groups, so certain kind of symmetry seems to be a relevant property in finding an appropriate $\gamma_{rd}(G)$-function. 
We observe that in both cases, the graphs are Cayley graphs over abelian groups, namely $Q_4\cong \Cay(\ZZ_4,\{e_1,e_2,e_3,e_4\})$ and $C_m[2K_1]\cong \Cay(\ZZ_m\times \ZZ_2,\{(\pm 1,0),(\pm 1,1)\})$. 
We can thus ask the following questions:
\begin{itemize}
\item Question 1: Are there any $d$-RDR graphs that are not obtained as Cayley graphs over some abelian group? 
\item 
Question 2: More generaly, are there any  $d$-RDR graphs that are not vertex transitive?
\end{itemize}
\section{Acknowledgments}
This work is supported in part by the Slovenian Research Agency (ARRS), research program P1-0285 and research projects J1-9108, J1-1694, J1-1695. The author also thanks to Primo\v z \v Sparl for suggesting the problem.

\end{document}